\documentclass[3p,12pt]{elsarticle}

\usepackage{amssymb,latexsym,amsmath,color,amsthm,url,multirow}

\journal{}

\newcommand{\eps}{\varepsilon}
\newcommand{\set}[1]{\left\{#1\right\}}

\newcommand{\p}{\partial}

\newcommand{\mH}{\mathbf{H}}

\newcommand{\me}{\mathbf{e}}

\newcommand{\mn}{\mathbf{n}}
\newcommand{\mt}{\mathbf{t}}
\newcommand{\mx}{\mathbf{x}}

\newcommand{\mz}{\mathbf{z}}
\newcommand{\ve}{\boldsymbol{\eta}}
\newcommand{\vn}{\boldsymbol{\nu}}

\newcommand{\vt}{\boldsymbol{\theta}}

\newtheorem{thm}{Theorem}[section]
\newtheorem{cor}[thm]{Corollary}
\newtheorem{lem}[thm]{Lemma}

\makeatletter
\newcommand*{\rom}[1]{\expandafter\@slowromancap\romannumeral #1@}
\makeatother

\begin{document}

\begin{frontmatter}



\title{Topological derivative-based technique for imaging thin inhomogeneities with few incident directions}

\author{Won-Kwang Park}
\ead{parkwk@kookmin.ac.kr}
\address{Department of Mathematics, Kookmin University, Seoul, 136-702, Korea.}

\begin{abstract}
Many non-iterative imaging algorithms require a large number of incident directions. Topological derivative-based imaging techniques can alleviate this problem, but lacks a theoretical background and a definite means of selecting the optimal incident directions. In this paper, we rigorously analyze the mathematical structure of a topological derivative imaging function, confirm why a small number of incident directions is sufficient, and explore the optimal configuration of these directions. To this end, we represent the topological derivative based imaging function as an infinite series of Bessel functions of integer order of the first kind. Our analysis is supported by the results of numerical simulations.\end{abstract}


\end{frontmatter}





\section{Introduction}
Inverse scattering problems identify certain characteristics of unknown targets embedded in a medium from the measured boundary data. For this purpose, researchers have developed various identification algorithms, most of which are based on Newton-type iteration schemes. Related works can be found in \cite{ADIM,B3,DL,LLS,PL4,S1,SZ,VXB,Z} and references therein. Although these schemes are regarded as promising techniques, they are not extendible to multiple-target identification. Furthermore, a good result requires additional regularization terms that largely depend on the specific problems, computation of the Fr{\'e}chet derivative, and \textit{a priori} information of the unknown targets. Moreover, if the initial guess is poorly chosen, the iteration procedure leads to more severe problems such as non-convergence, local (rather than global) minimization, and slow convergence to a solution (which incurs high computational cost). Generally, the success of iteration-based schemes highly depends on a good initial guess.

As an alternative, various non-iterative detection techniques have been investigated. Examples are multiple signal classification (MUSIC), the linear sampling method, Kirchhoff and subspace migrations, and inverse Fourier transform-based algorithms. Unfortunately, these methods yield inaccurate results, and require a large number of directions of the incident and scattered fields \cite{AK2}. The topological derivative strategy is a non-iterative imaging technique that has been successfully applied to various inverse scattering problems. The topological derivative can accurately replicate the true shape of the unknown target, even with few directions of the incident field data \cite{AGJK,P-TD1,P-TD3}. However, this advantage has mostly been confirmed through numerical simulations. The development of an appropriate mathematical theory remains an interesting and worthwhile task.

The present paper analyzes the mathematical structure of a topological derivative-based imaging function with a small number of directions of the incident fields. The function is applied to an arbitrarily shaped, thin penetrable inhomogeneity. Under thin-inhomogeneity conditions, uniform convergence of the Jacobi-Anger expansion formula and the asymptotic properties of Bessel functions, the far-field pattern can be represented by an asymptotic expansion formula. Therefore, we show that the imaging function can be represented as an infinite series of Bessel functions of integer order of the first kind. From the derived structure of the imaging function, we confirm that to ensure a good result, we require an even number of directions (at least $4$); moreover, these directions must be symmetric.

The remainder of this paper is organized as follows. Section \ref{sec:2} briefly introduces the two-dimensional direct scattering problem and the topological derivative. Section \ref{sec:3} explores the structures of the single- and multi-frequency topological derivative imaging functions with a small number of incident directions. In Section \ref{sec:4}, our investigations are supported by the results of numerical simulations with noisy data. Section \ref{sec:5} presents a short conclusion outlining the current work and suggesting ideas for future work.

\section{Direct scattering problem and topological derivative}\label{sec:2}
This section briefly introduces the two-dimensional direct scattering problem for a thin penetrable inhomogeneity, and presents the normalized topological derivative imaging function. A more detailed description is given in \cite{P-TD1,P-TD3}.
\subsection{Two-dimensional direct scattering problem}
Let $\Omega\subset\mathbb{R}^2$ be a homogeneous domain with a smooth boundary $\p\Omega$, which is a $\mathcal{C}^3$ curve. This domain contains a thin, curve-like homogeneous electromagnetic inhomogeneity. Let us assume that this thin inhomogeneity (denoted as $\Gamma$) resides in the neighborhood of a simple smooth curve $\sigma:=\sigma(\mx)$ as
\[\Gamma=\set{\mx+\gamma\mn(\mx):\mx\in\sigma,~\gamma\in(-h,h)},\]
where $\mn(\mx)$ is the unit normal to $\sigma$ at $\mx$ and $h$ is a positive constant denoting the thickness of $\Gamma$ (see Figure \ref{FigureGamma}). Throughout this paper, we assume that the applied frequency of a given wavelength $\lambda$ is $\omega=\frac{2\pi}{\lambda}$, and that the thickness $h$ of $\Gamma$ is sufficiently smaller than $\lambda$ ($h\ll\lambda$). We also assume that the inhomogeneity is located at some distance from the boundary $\p\Omega$, and never touches that boundary. In other words, there is a nonzero positive constant $s$ such that
\[\mbox{dist}(\sigma,\p\Omega)=s\gg h.\]

\begin{figure}
\begin{center}
\includegraphics[width=0.4\textwidth,keepaspectratio=true,angle=0]{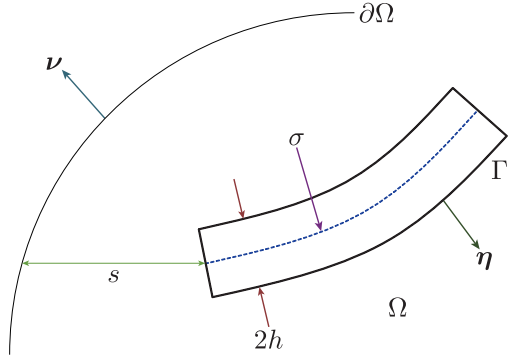}
\caption{\label{FigureGamma}Thin penetrable inhomogeneity $\Gamma$ of thickness $2h$ embedded in $\Omega\subset\mathbb{R}^2$.}
\end{center}
\end{figure}

Let every material be classified by its dielectric permittivity and magnetic permeability at a given frequency $\omega$. Let the permittivity and permeability be $0<\eps_0<+\infty$ and $0<\mu_0<+\infty$ respectively in the domain $\Omega$, and $0<\eps<+\infty$ and $0<\mu<+\infty$ respectively in the inhomogeneity $\Gamma$. We then define the piecewise constant dielectric permittivity $\eps(\mx)$ and magnetic permeability $\mu(\mx)$ as
\begin{equation}\label{EPSMU}
\eps(\mx)=\left\{\begin{array}{ccl}
\eps_0&\mbox{for}&\mx\in\Omega\backslash\overline{\Gamma}\\
\eps&\mbox{for}&\mx\in\Gamma
\end{array}\right.
\quad\mbox{and}\quad
\mu(\mx)=\left\{\begin{array}{ccl}
\mu_0&\mbox{for}&\mx\in\Omega\backslash\overline{\Gamma}\\
\mu&\mbox{for}&\mx\in\Gamma,
\end{array}\right.
\end{equation}
respectively. For simplicity, we set $\eps_0=\mu_0=1$, $\eps>\eps_0$, and $\mu>\mu_0$.

When $\Gamma$ exists, let $u^{(n)}(\mx;\omega)$ be the time-harmonic total field satisfying the Helmholtz equation at frequency $\omega$. Then we have
\begin{equation}\label{ForwardProblem}
\left\{\begin{array}{rcl}
\displaystyle\nabla\cdot\left(\frac{1}{\mu(\mx)}\nabla u^{(n)}(\mx;\omega)\right)+\omega^2\eps(\mx)u^{(n)}(\mx;\omega)=0&\mbox{in}&\Omega\\
\noalign{\medskip}\displaystyle\frac{1}{\mu_0}\frac{\p u^{(n)}(\mx;\omega)}{\p\vn(\mx)}=\frac{\p e^{i\omega\vt_n\cdot\mx}}{\p\vn(\mx)}=g^{(n)}(\mx;\omega)\in L^2(\partial\Omega)&\mbox{on}&\p\Omega,\\
\end{array}\right.
\end{equation}
with transmission conditions
\[u^{(n)}(\mx;\omega)|_+=u^{(n)}(\mx;\omega)|_-\quad\mbox{and}\quad
\frac{1}{\mu_0}\frac{\p u^{(n)}(\mx;\omega)}{\p \ve(\mx)}\bigg|_+=\frac{1}{\mu}\frac{\p u^{(n)}(\mx;\omega)}{\p \ve(\mx)}\bigg|_-\quad\mbox{on}\quad\p\Gamma.\]
Here, $\vn(\mx)$ and $\ve(\mx)$ represent the unit outward normal to $\mx\in\p\Omega$ and $\mx\in\p\Gamma$, respectively. The subscript $\pm$ denotes the limiting values as
\begin{align*}
  u^{(n)}(\mx;\omega)|_\pm&=\lim_{t\to0+}u^{(n)}(\mx\pm t\ve(\mx);\omega)\\
  \frac{\p u^{(n)}(\mx;\omega)}{\p \ve(\mx)}\bigg|_\pm&=\lim_{t\to0+}\nabla u^{(n)}(\mx\pm t\ve(\mx);\omega)\cdot\ve(\mx)
\end{align*}
for $\mx\in\p\Gamma$, and $\vt_n$ denotes a two-dimensional vector on the unit circle $\mathbb{S}^1$ for $n=1,2,\cdots,N$.

Similarly, in the absence of $\Gamma$, let $u_{\mathrm{B}}^{(n)}(\mx;\omega)=e^{i\omega\vt_n\cdot\mx}$ denote a field satisfying (\ref{ForwardProblem}). This is the background solution. Throughout this paper, we assume that $\omega^2$ is not an eigenvalue of (\ref{ForwardProblem}).

\subsection{Review of normalized topological derivative based imaging function}
In this section, we introduce the basic concept of the topological derivative operated at a fixed single frequency. For detailed discussions, the reader is referred to \cite{AGJK,AKLP,B1,CR,EKS,JSV,P-TD1,P-TD3,SZ}. Let $u^{(n)}(\mx;\omega)$ and $u_{\mathrm{B}}^{(n)}(\mx;\omega)$ be the total and background solutions of (\ref{ForwardProblem}), respectively. To find the shape of $\Gamma$, we consider the following energy function, which depends on the solution $u^{(n)}(\mx;\omega)$:
\begin{equation}\label{Energy}
  \mathbb{E}(\Omega;\omega):=\frac12\sum_{n=1}^{N}\|u^{(n)}(\mx;\omega)-u_{\mathrm{B}}^{(n)}(\mx;\omega)\|_{L^2(\partial\Omega)}^2=\frac12\sum_{n=1}^{N}\int_{\p\Omega}|u^{(n)}(\mx;\omega)-u_{\mathrm{B}}^{(n)}(\mx;\omega)|^2dS(\mx).
\end{equation}

Assume that an electromagnetic inhomogeneity $\Sigma$ of small diameter $r$ is created at a certain position $\mz\in\Omega\backslash\partial\Omega$. Let $\Omega|\Sigma$ denote the domain of this position. As the inhomogeneity changes the topology of the entire domain, we can consider the corresponding topological derivative $d_T\mathbb{E}(\mz)$ on $\mathbb{E}(\Omega)$ with respect to point $\mz$ as
\begin{equation}\label{TopDerivative}
  d_T\mathbb{E}_{\mathrm{MF}}(\mz;\omega)=\lim_{r\to0+}\frac{\mathbb{E}(\Omega|\Sigma;\omega) -\mathbb{E}(\Omega;\omega)}{\varphi(r;\omega)},
\end{equation}
where $\varphi(r;\omega)\longrightarrow0$ as $r\longrightarrow0+$. From (\ref{TopDerivative}), we obtain the asymptotic expansion:
\begin{equation}\label{AsymptoticExpansionTopologicalDerivative}
  \mathbb{E}(\Omega|\Sigma;\omega)=\mathbb{E}(\Omega;\omega)+\varphi(r;\omega)d_T\mathbb{E}_{\mathrm{MF}}(\mz;\omega)+o(\varphi(r;\omega)).
\end{equation}

The following normalized topological derivative imaging function $\mathbb{E}_{\mathrm{SF}}(\mz;\omega)$ was introduced in \cite{P-TD3}:
\begin{equation}\label{NormalizedTopologicalDerivative}
  \mathbb{E}_{\mathrm{SF}}(\mz;\omega)=\frac{1}{2}\bigg(\frac{d_T\mathbb{E}_\eps(\mz;\omega)}{\max[d_T\mathbb{E}_\eps(\mz;\omega)]}+\frac{d_T\mathbb{E}_\mu(\mz;\omega)}{\max[d_T\mathbb{E}_\mu(\mz;\omega)]}\bigg).
\end{equation}
In the cases of purely dielectric permittivity contrast ($\eps\ne\eps_0$ and $\mu=\mu_0$) and magnetic permeability contrast ($\eps=\eps_0$ and $\mu\ne\mu_0$), the $d_T\mathbb{E}_\eps(\mz;\omega)$ and $d_T\mathbb{E}_\mu(\mz;\omega)$ satisfying (\ref{AsymptoticExpansionTopologicalDerivative}) are respectively given by (see \cite{P-TD3})
  \begin{align}
      d_T\mathbb{E}_\eps(\mz;\omega)&=\mathrm{Re}\sum_{n=1}^{N}\bigg(u_{\mathrm{A}}^{(n)}(\mz;\omega)\overline{u_{\mathrm{B}}^{(n)}(\mz;\omega)}\bigg),\label{TopologicalDerivative1}\\
      d_T\mathbb{E}_\mu(\mz;\omega)&=-\mathrm{Re}\sum_{n=1}^{N}\bigg(\nabla u_{\mathrm{A}}^{(n)}(\mz;\omega)\cdot\overline{\nabla u_{\mathrm{B}}^{(n)}(\mz;\omega)}\bigg),\label{TopologicalDerivative2}
  \end{align}
where $u_{\mathrm{A}}^{(n)}(\mx;\omega)$ satisfies the adjoint problem
\begin{equation}\label{Adjoint1}
  \left\{\begin{array}{rcl}
    \displaystyle\Delta u_{\mathrm{A}}^{(n)}(\mx;\omega)+\omega^2u_{\mathrm{A}}^{(n)}(\mx;\omega)=0&\mbox{in}&\Omega\\
    \noalign{\medskip}\displaystyle\frac{1}{\mu_0}\frac{\p u_{\mathrm{A}}^{(n)}(\mx;\omega)}{\p\boldsymbol{\nu}(\mx)}=u^{(n)}(\mx;\omega)-u_{\mathrm{B}}^{(n)}(\mx;\omega)&\mbox{on}&\p\Omega.
  \end{array}\right.
\end{equation}

We now analyze the properties of (\ref{MultiFrequencyTopologicalDerivative}). The following result from \cite{P-TD3} plays an important role in our analysis.

\begin{lem}\label{lemma}
  Let $A\sim B$ imply the existence of some constant $C$ such that $A=BC$, and let $\mathrm{Re}(f)$ denote the real part of $f$. Then, (\ref{TopologicalDerivative1}) and (\ref{TopologicalDerivative2}) satisfy
  \begin{align*}
    d_T\mathbb{E}_\eps(\mz;\omega)&\sim\mathrm{Re}\sum_{n=1}^{N}\int_\sigma(\eps-\eps_0) e^{i\omega\vt_n\cdot(\mx-\mz)}d\sigma(\mx)\\
    d_T\mathbb{E}_\mu(\mz;\omega)&\sim\mathrm{Re}\sum_{n=1}^{N}\int_\sigma \bigg[2\bigg(\frac{1}{\mu}-\frac{1}{\mu_0}\bigg)\vt_n\cdot\mt(\mx)
    +2\bigg(\frac{1}{\mu_0}-\frac{\mu}{\mu_0^2}\bigg)\vt_n\cdot\mn(\mx)\bigg] e^{i\omega\vt_n\cdot(\mx-\mz)}d\sigma(\mx),
  \end{align*}
  where $\mt(\mx)$ and $\mn(\mx)$ are unit vectors that are respectively tangent and normal to the supporting curve $\sigma$ at $\mx$.
\end{lem}

Recent works \cite{P-TD3,AGKPS,P-SUB3,GP} have confirmed that multi-frequency applications guarantee better results than single-frequency applications. For this reason, we consider the following normalized multi-frequency based topological derivative imaging function. For several frequencies $\set{\omega_k:k=1,2,\cdots,K}$, we define
\begin{equation}\label{MultiFrequencyTopologicalDerivative}
  \mathbb{E}_{\mathrm{MF}}(\mz;K):=\frac{1}{K}\sum_{k=1}^{K}\mathbb{E}_{\mathrm{SF}}(\mz;\omega_k)=
  \frac{1}{2K}\sum_{k=1}^{K}\bigg(\frac{d_T\mathbb{E}_\eps(\mz;\omega_k)}{\max[d_T\mathbb{E}_\eps(\mz;\omega_k)]} +\frac{d_T\mathbb{E}_\mu(\mz;\omega_k)}{\max[d_T\mathbb{E}_\mu(\mz;\omega_k)]}\bigg),
\end{equation}
where $d_T\mathbb{E}_\eps(\mz;\omega_k)$ and $d_T\mathbb{E}_\mu(\mz;\omega_k)$ satisfy (\ref{TopologicalDerivative1}) and (\ref{TopologicalDerivative2}), respectively, at $\omega=\omega_k$ and $k=1,2,\cdots,K$.

\section{Analysis of imaging function with small number of incident directions}\label{sec:3}
In this section, we formalize (\ref{MultiFrequencyTopologicalDerivative}) through Lemma \ref{lemma}. Throughout this section, we assume the following form of $\vt_n$:
\begin{equation}\label{Direction}
\vt_n=[\cos(\theta_n),\sin(\theta_n)]^T=\left[\cos\left(\frac{2\pi(n-1)}{N}\right),\sin\left(\frac{2\pi(n-1)}{N}\right)\right]^T\quad\mbox{for}\quad n=1,2,\cdots,N,
\end{equation}
where the total number of incident directions $N$ is small. We then obtain the following main result.

\begin{thm}\label{Theorem}When $N$ is sufficiently small, we have
  \begin{multline}\label{Structure1}
  \mathbb{E}_{\mathrm{SF}}(\mz;\omega)\sim\sum_{n=1}^{N}\int_\sigma\bigg[(\eps-\eps_0)+2\bigg(\frac{1}{\mu}-\frac{1}{\mu_0}\bigg)\vt_n\cdot\mt(\mx)
    +2\bigg(\frac{1}{\mu_0}-\frac{\mu}{\mu_0^2}\bigg)\vt_n\cdot\mn(\mx)\bigg]\\
    \times\left(J_0(\omega|\mx-\mz|)
 +2\sum_{m=1}^{\infty}(-1)^mJ_{2 m}(\omega|\mx-\mz|)\cos\{2 m(\theta_n-\phi_\mz)\}\right)d\sigma(\mx)
 \end{multline}
  and
  \begin{multline}\label{Structure2}
  \mathbb{E}_{\mathrm{MF}}(\mz;K)\sim\sum_{n=1}^{N}\int_\sigma\bigg[(\eps-\eps_0)+2\bigg(\frac{1}{\mu}-\frac{1}{\mu_0}\bigg)\vt_n\cdot\mt(\mx)
    +2\bigg(\frac{1}{\mu_0}-\frac{\mu}{\mu_0^2}\bigg)\vt_n\cdot\mn(\mx)\bigg]\\
    \times\frac{1}{\omega_K-\omega_1}\left(\Lambda(|\mx-\mz|;K)
 +2\int_{\omega_1}^{\omega_K}\sum_{m=1}^{\infty}(-1)^mJ_{2 m}(\omega|\mx-\mz|)\cos\{2 m(\theta_n-\phi_\mz)\}d\omega\right)d\sigma(\mx),
 \end{multline}
 Here, $\Lambda(x;K)$ is defined as
\begin{multline}\label{FunctionStruve}
  \Lambda(x;K):=\omega_KJ_0(\omega_Kx)+\frac{\omega\pi}{2}\bigg(J_1(\omega_Kx)\mH_0(\omega_Kx)-J_0(\omega_Kx)\mH_1(\omega_Kx)\bigg)\\
  -\omega_1J_0(\omega_1x)+\frac{\omega\pi}{2}\bigg(J_1(\omega_1x)\mH_0(\omega_1x)-J_0(\omega_1x)\mH_1(\omega_1x)\bigg),
\end{multline}
where $J_n(x)$ denotes the Bessel function of order $n$ of the first kind and $\mH_n(x)$ denotes the Struve function of order $n$ (see \cite{AS-Book}).
\end{thm}
\begin{proof}
 First, let us consider the following term:
 \[\frac{d_T\mathbb{E}_\eps(\mz;\omega)}{\max[d_T\mathbb{E}_\eps(\mz;\omega)]}\approx\mathrm{Re}\bigg(\sum_{n=1}^{N}\int_\sigma(\eps-\eps_0)e^{i\omega\vt_n\cdot(\mx-\mz)}d\sigma(\mx)\bigg).\]
As $N$ is small, the above formula cannot be expressed in integral form, so we must find an alternative representation. Setting $\mx-\mz=r[\cos(\phi_\mz),\sin(\phi_\mz)]^T$, we have $\vt_n\cdot(\mx-\mz)=r\cos(\theta_n-\phi_\mz)$, and the following Jacobi-Anger expansion holds uniformly:
 \begin{equation}\label{JacobiAnger}
   e^{ir\cos\theta}=\sum_{m=-\infty}^{\infty}i^mJ_m(r)e^{im\phi_\mz}
 \end{equation}
 By elementary calculus, we can evaluate
 \begin{multline*}
   \int_\sigma(\eps-\eps_0)e^{i\omega\vt_n\cdot(\mx-\mz)}d\sigma(\mx)=\int_\sigma(\eps-\eps_0)\sum_{m=-\infty}^{\infty}i^mJ_m(\omega|\mx-\mz|)e^{im(\theta_n-\phi_\mz)}d\sigma(\mx)\\
   =\int_\sigma(\eps-\eps_0)\left(J_0(\omega|\mx-\mz|)
 +2\sum_{m=1}^{\infty}i^mJ_m(\omega|\mx-\mz|)\cos\{m(\theta_n-\phi_\mz)\}\right)d\sigma(\mx).
 \end{multline*}
 Hence,
 \begin{multline}\label{term1}
   \frac{d_T\mathbb{E}_\eps(\mz;\omega)}{\max[d_T\mathbb{E}_\eps(\mz;\omega)]}\approx\left.\sum_{n=1}^{N}\int_\sigma(\eps-\eps_0)\right(J_0(\omega|\mx-\mz|)\\
 \left.+2\sum_{m=1}^{\infty}(-1)^mJ_{2 m}(\omega|\mx-\mz|)\cos\{2 m(\theta_n-\phi_\mz)\}\right)d\sigma(\mx).
 \end{multline}
 Similarly, we can evaluate
 \begin{multline}\label{term2}
   \frac{d_T\mathbb{E}_\mu(\mz;\omega)}{\max[d_T\mathbb{E}_\mu(\mz;\omega)]}=\sum_{n=1}^{N}\int_\sigma\bigg[2\bigg(\frac{1}{\mu}-\frac{1}{\mu_0}\bigg)\vt_n\cdot\mt(\mx)+2\bigg(\frac{1}{\mu_0}-\frac{\mu}{\mu_0^2}\bigg)\vt_n\cdot\mn(\mx)\bigg]\\
   \times\left(J_0(\omega|\mx-\mz|)+2\sum_{m=1}^{\infty}(-1)^mJ_{2 m}(\omega|\mx-\mz|)\cos\{2 m(\theta_n-\phi_\mz)\}\right)d\sigma(\mx).
 \end{multline}
 Combining (\ref{term1}) and (\ref{term2}), we obtain (\ref{Structure1}).

Equation (\ref{Structure2}) is then easily obtained from (\ref{term1}), (\ref{term2}) and the following indefinite integration
 \[\int J_0(x)dx=xJ_0(x)+\frac{\pi x}{2}\bigg(J_1(x)\mH_0(x)-J_0(x)\mH_1(x)\bigg),\].
This completes the proof.
\end{proof}

From the  structure identified in Theorem \ref{Theorem}, we obtain the following result. The proof is very similar to that of \cite[Theorem 3.4]{AJMP}.

\begin{cor}\label{Corollary}
  Let
  \[\mathbb{D}_{\mathrm{SF}}(\mx,\mz;\omega):=\sum_{m=1}^{\infty}(-1)^mJ_{2 m}(\omega|\mx-\mz|)\cos\{2 m(\theta_n-\phi_\mz)\}\]
  and
  \[\mathbb{D}_{\mathrm{MF}}(\mx,\mz;K)=\frac{1}{\omega_K-\omega_1}\int_{\omega_1}^{\omega_K}\sum_{m=1}^{\infty}(-1)^mJ_{2 m}(\omega|\mx-\mz|)\cos\{2 m(\theta_n-\phi_\mz)\}d\omega.\]
 If $M\in\mathbb{N}$ is sufficiently large and $\mz$ is close to $\mx$ such that $0<\omega_k|\mx-\mz|\ll\sqrt{M+1}$, then
  \[|\mathbb{D}_{\mathrm{MF}}(\mx,\mz;K)|\ll|\mathbb{D}_{\mathrm{SF}}(\mx,\mz;\omega_K)|\].
Conversely, if $\mz$ is far from $\mx$ such that $\omega_k|\mx-\mz|\gg|M^2-0.25|$, then
  \[|\mathbb{D}_{\mathrm{MF}}(\mx,\mz;K)|\ll|\mathbb{D}_{\mathrm{SF}}(\mx,\mz;\omega_1)|.\]
\end{cor}

From the results derived in Theorem \ref{Theorem}, we observe that $\mathbb{E}_{\mathrm{SF}}(\mz;\omega)\approx1$ and $\mathbb{E}_{\mathrm{MF}}(\mz;K)\approx1$ at $\mz\in\Gamma$. However, owing to the small oscillation of $\mathbb{E}_{\mathrm{MF}}(\mz;K)$ and the term $\mathbb{D}_{\mathrm{MF}}(\mx,\mz;K)$, which disturbs the imaging performance less than $\mathbb{D}_{\mathrm{SF}}(\mx,\mz;\omega)$, the multi-frequency application should guarantee better imaging performance than single-frequency application.

Now, let us consider the conditions that guarantee good results. Eliminating the disturbance terms
\[\sum_{n=1}^{N}\sum_{m=1}^{\infty}(-1)^mJ_{2 m}(\omega|\mx-\mz|)\cos\{2 m(\theta_n-\phi_\mz)\}\]
or
\[\sum_{k=1}^{K}\left(\sum_{n=1}^{N}\sum_{m=1}^{\infty}(-1)^mJ_{2 m}(\omega_k|\mx-\mz|)\cos\{2 m(\theta_n-\phi_\mz)\}\right)\]
from the identified structures (\ref{Structure1}) and (\ref{Structure2}) will improve the imaging result. In the simplest approach, we set $\omega$ to $+\infty$. The resulting asymptotic form of the Bessel function
\[J_{2 m}(\omega|\mx-\mz|)\approx\sqrt{\frac{2}{\pi\omega|\mx-\mz|}}\cos\left(\omega|\mx-\mz|-m\pi-\frac{\pi}{4}\right)\longrightarrow0,\]
eliminates the disturbance terms. However, this is an ideal condition, and another elimination approach is needed in practice.

From the above observation, and the lack of \textit{a priori} information of $\Gamma?s$ location, we conclude that we cannot control $J_{2 m}(\omega|\mx-\mz|)$ for any $\omega$. Therefore, to eliminate the disturbance terms, we focus on the term $\cos\{2 m(\theta_n-\phi_\mz)\}$, where $n=1,2,\cdots,N$. The value of $\phi_\mz$ is of course unknown, but the following properties
\[\cos\{2 m(\theta_n-\phi_\mz)\}=\cos(2 m\theta_n)\cos(2 m\phi_\mz)+\sin(2 m\theta_n)\sin(2 m\phi_\mz)\]
and
\[\cos(\pi+\theta_n)=-\cos(\theta_n),\quad\sin(\pi+\theta_n)=-\sin(\theta_n),\]
imply that symmetric incident directions will guarantee superior results. Consequently, the total number of incident directions $N$ must be even, say $N=2L$, and $\vt_n$ in (\ref{Direction}) must satisfy $\vt_n=-\vt_{n+L}$ for $n=1,2,\cdots,L$.

From the above discussion, we conclude that for each applied incident direction $\vt_1$, the opposite direction $-\vt_1$ must also be applied. However, as these vectors do not span $\mathbb{S}^1$, another incident direction $\vt_2$, which is independent of $\vt_1$ (and its correspondingly $-\vt_2$), is required. To guarantee a good result, we need at least $4$ incident directions. Moreover, the directions $\vt_n$ must be distributed uniformly on $\mathbb{S}^1$. Thus, setting $\pm\vt_1=\pm\me_1=[\pm1,0]^T$ and $\pm\vt_2=\pm\me_2=[0,\pm1]^T$ should achieve successful imaging performance.

\section{Simulation results}\label{sec:4}
In this section, the derivations of Section \ref{sec:3} are supported by the results of numerical simulations. The simulated homogeneous domain $\Omega$ was a unit circle centered at the origin in $\mathbb{R}^2$, and the supporting curves of the thin inhomogeneities $\Gamma_j$ were described by two $\sigma_js$:
\begin{align}
\begin{aligned}\label{SupportingCurve}
\sigma_1&=\set{[s-0.2,-0.5s^2+0.5]^T:-0.5\leq s\leq0.5}\\
\sigma_2&=\set{[s+0.2,s^3+s^2-0.6]^T:-0.5\leq s\leq0.5}.
\end{aligned}
\end{align}
The thickness $h$ of the thin inhomogeneity $\Gamma_j$ was set to $0.02$, and the parameters $\eps_0$ and $\mu_0$ were set to $1$. In this section, we denote the permittivity and permeability of $\Gamma_j$ as $\eps_j$ and $\mu_j$, respectively, where $j=1,2$. The applied frequency was $\omega_k=2\pi/\lambda_k$ at wavelength $\lambda_k$, $k=1,2,\cdots,K(=10)$. We set $\omega=2\pi/0.5$ for single-frequency imaging and  $\lambda_1=0.7$ and $\lambda_{K}=0.3$ for multi-frequency imaging. To demonstrate the robustness of the proposed algorithm, we added a white Gaussian noise with a signal-to-noise ratio (SNR) of $20$dB to the unperturbed boundary data $u^{(l)}(\mathbf{x};\omega_k)$. The noise was imposed by the standard MATLAB command `awgn'. In this section, we consider both permittivity and permeability contrast with $\eps_j=\mu_j=5$ for $j=1,2,$ and $3$.

We first consider the imaging of $\Gamma_1$. Maps of $\mathbb{E}_{\mathrm{SF}}(\mz;\omega)$ for $N=4$ and $N=5$ are displayed in Figure \ref{Gamma1-Single}. Note that when $N$ is small and odd,  $\mathbb{E}_{\mathrm{SF}}(\mz;\omega)$ yields very poor results. Although the results are vastly improved when $N$ is even, the map of $\mathbb{E}_{\mathrm{SF}}(\mz;\omega)$ cannot identify the shape outline of $\Gamma_1$ unless $N$ is increased.

\begin{figure}[h]
\begin{center}
\includegraphics[width=0.325\textwidth]{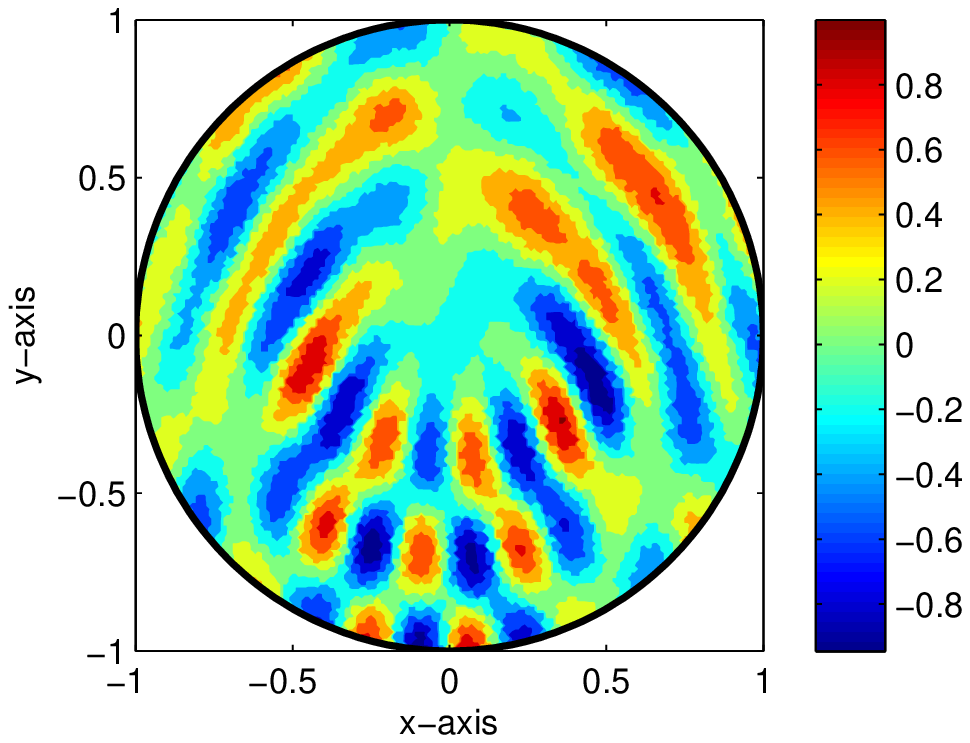}
\includegraphics[width=0.325\textwidth]{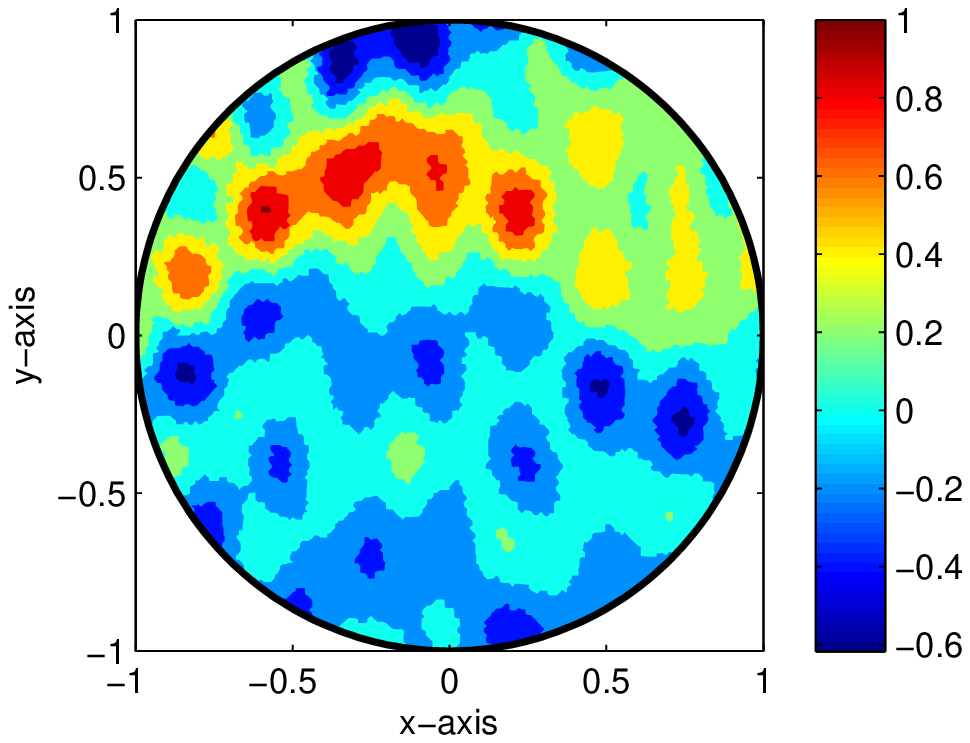}
\includegraphics[width=0.325\textwidth]{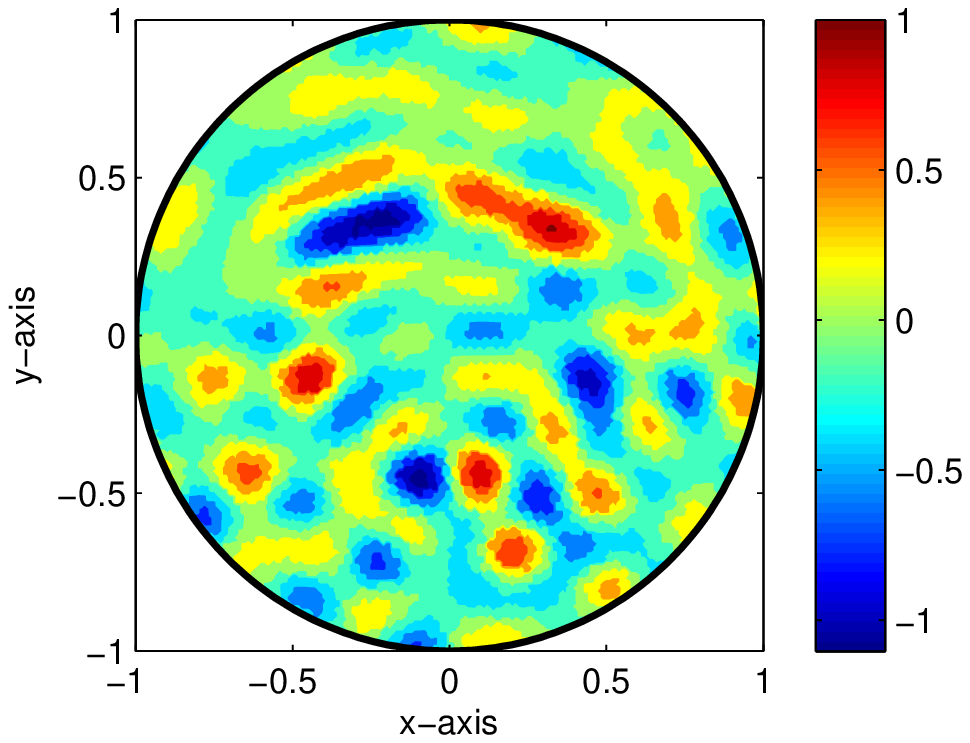}\\
\includegraphics[width=0.325\textwidth]{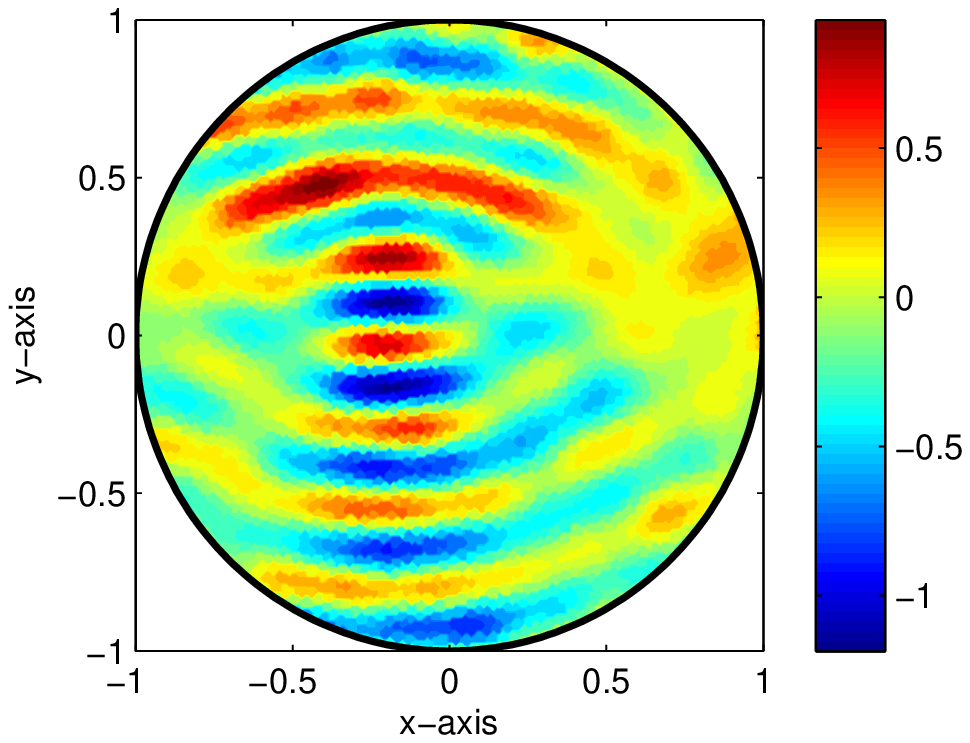}
\includegraphics[width=0.325\textwidth]{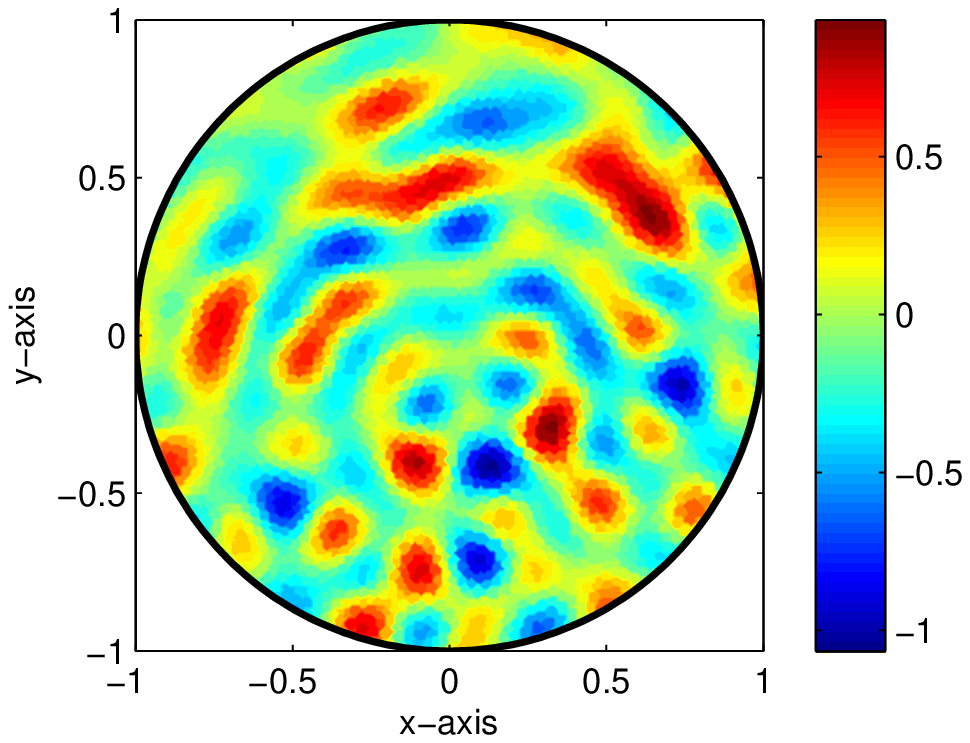}
\includegraphics[width=0.325\textwidth]{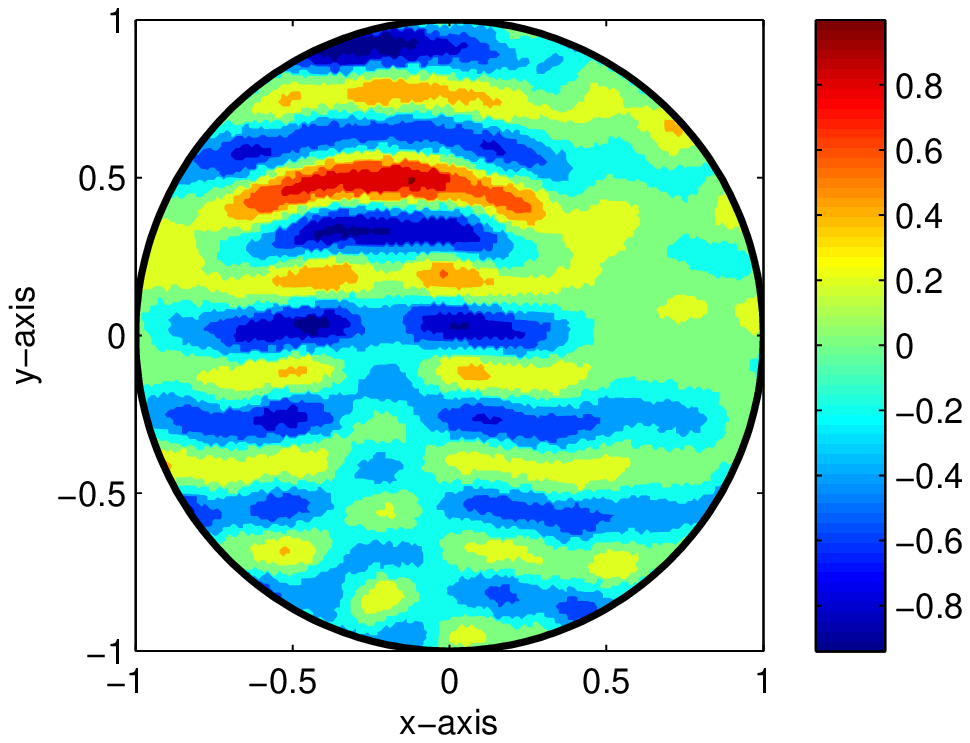}
\caption{Maps of $\mathbb{E}_{\mathrm{SF}}(\mz;2\pi/0.5)$ for $N=1$ (top left), $N=2$ (top right), $N=3$ (middle left), $N=4$ (middle right), $N=5$ (bottom left), and $N=6$ (bottom right) on a thin inhomogeneity of $\Gamma_1$.}\label{Gamma1-Single}
\end{center}
\end{figure}

Recent work \cite{P-TD1} has shown that when $N$ is small and $K$ is sufficiently large, the target shapes can be recognized from the map of $\mathbb{E}_{\mathrm{MF}}(\mz;K)$. Figure \ref{Gamma1-Multi} shows maps of $\mathbb{E}_{\mathrm{MF}}(\mz;10)$ for $N=4$ and $N=5$. As expected, the results accurately capture the true shape of $\Gamma_1$. Interestingly, when $N=5$, the shape of $\Gamma_1$ is clearly outlined despite the large number of artifacts.

\begin{figure}[h]
\begin{center}
\includegraphics[width=0.325\textwidth]{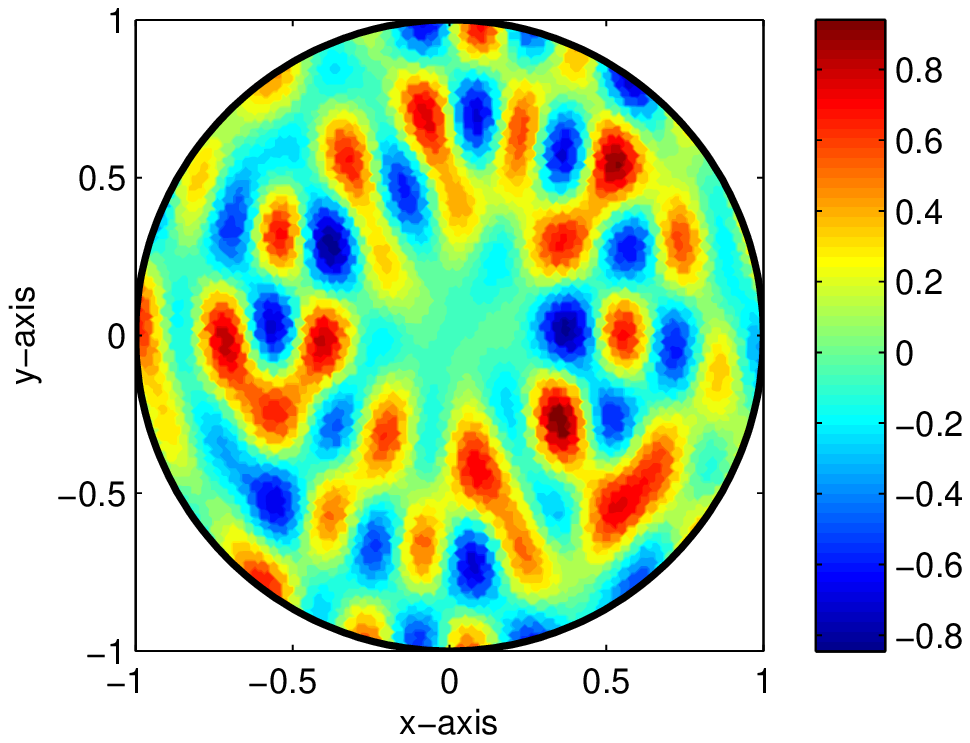}
\includegraphics[width=0.325\textwidth]{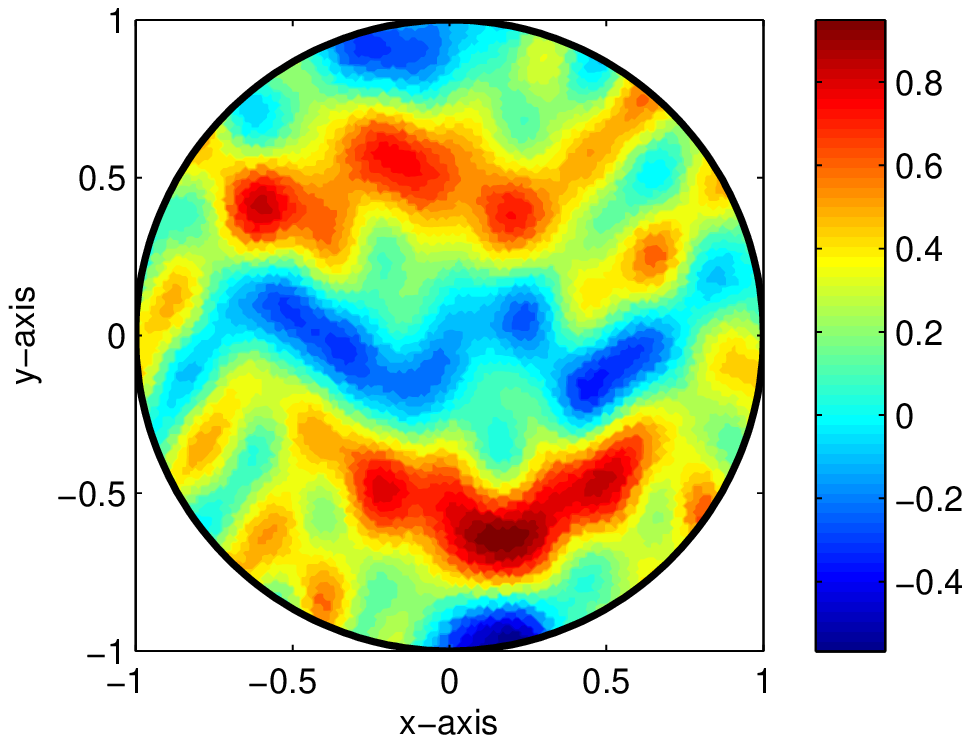}
\includegraphics[width=0.325\textwidth]{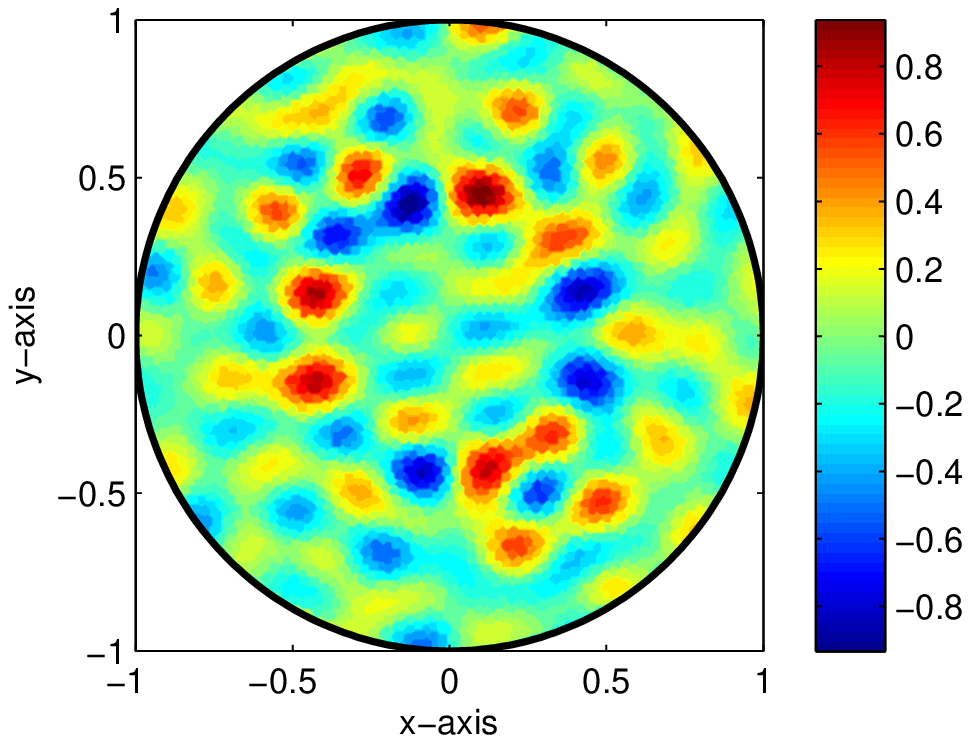}\\
\includegraphics[width=0.325\textwidth]{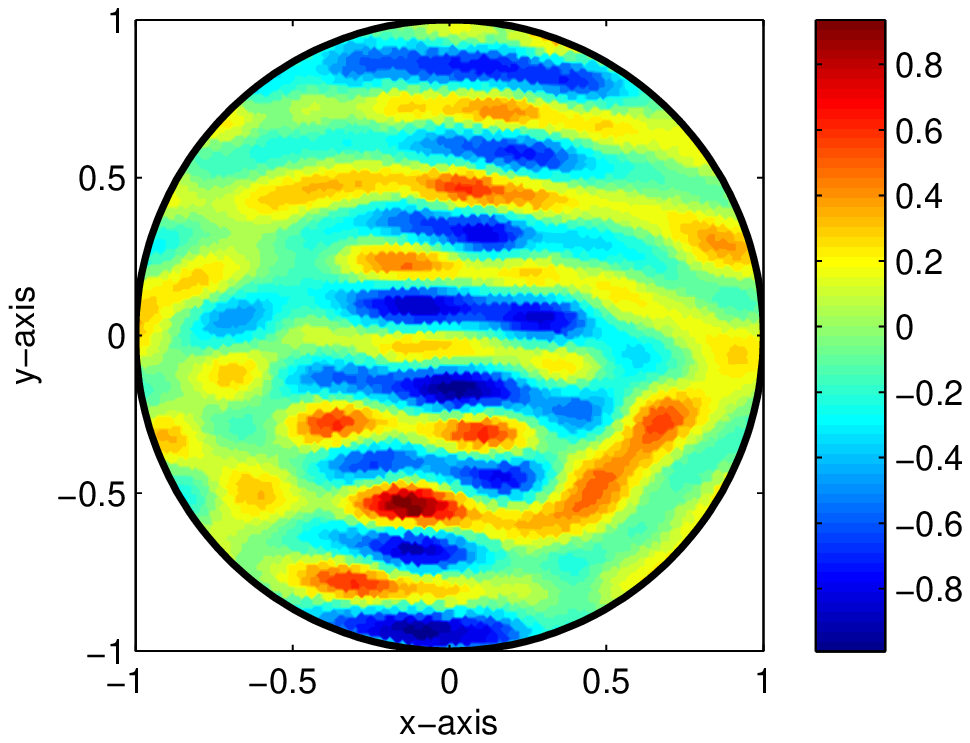}
\includegraphics[width=0.325\textwidth]{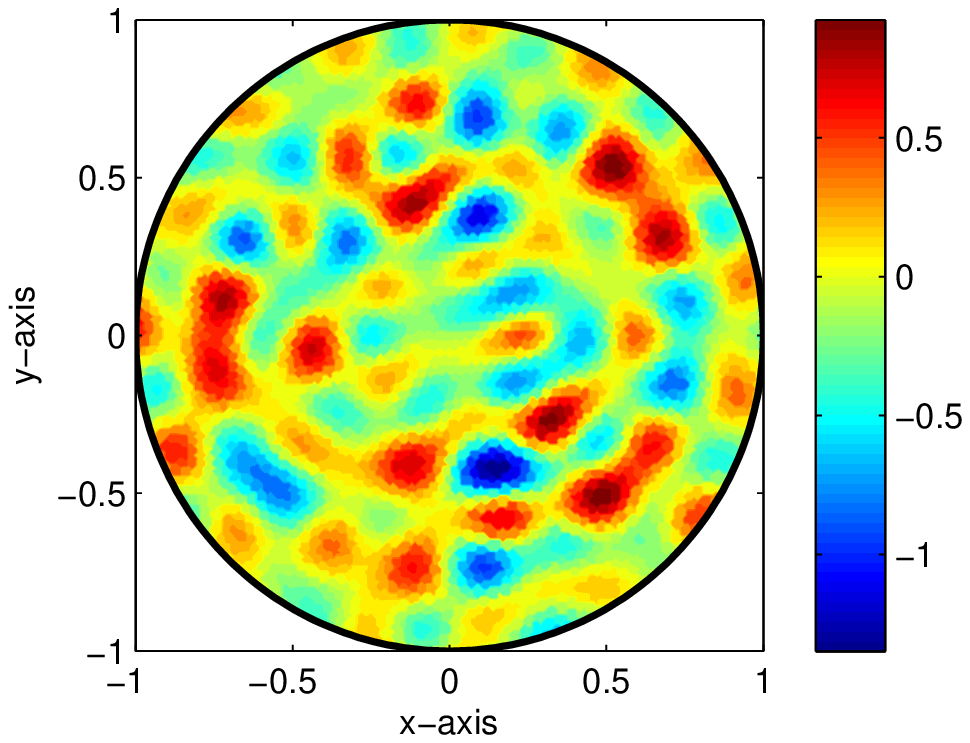}
\includegraphics[width=0.325\textwidth]{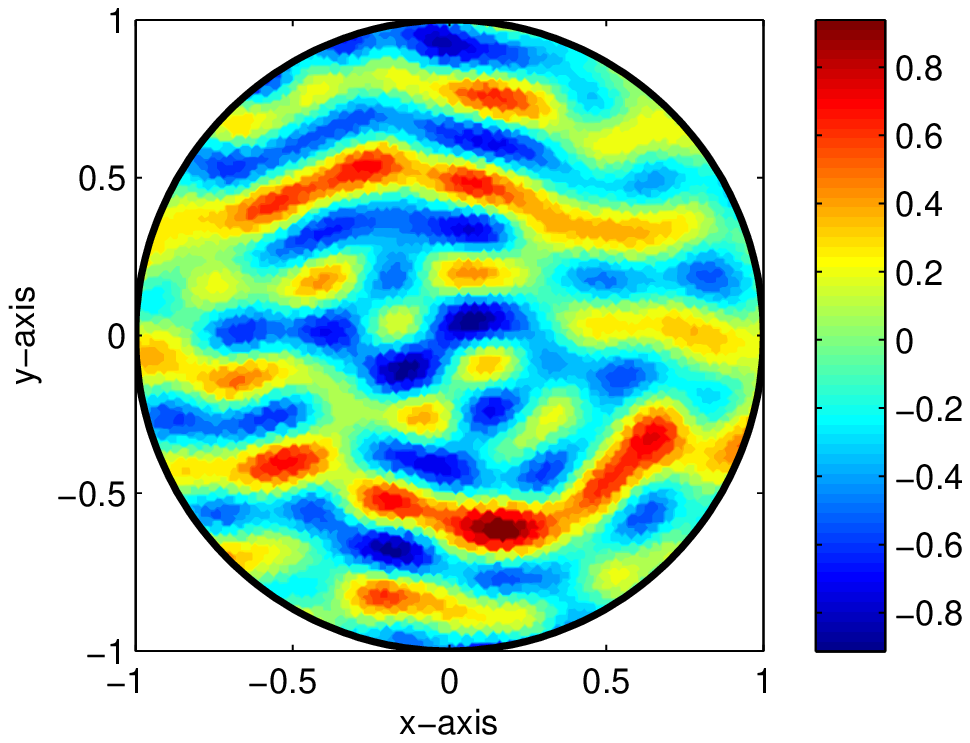}
\caption{Maps of $\mathbb{E}_{\mathrm{MF}}(\mz;10)$ for $N=1$ (top, left), $N=2$ (top, right), $N=3$ (middle, left), $N=4$ (middle, right), $N=5$ (bottom, left), and $N=6$ (bottom, right) on a thin inhomogeneity of $\Gamma_1\cup\Gamma_2$.}\label{Gamma1-Multi}
\end{center}
\end{figure}

Maps of $\mathbb{E}_{\mathrm{SF}}(\mz;\omega)$ and $\mathbb{E}_{\mathrm{MF}}(\mz;K)$ on a thin inhomogeneity of $\Gamma_2$ are displayed in Figs. \ref{Gamma2-Single} and \ref{Gamma2-Multi}, respectively. The imaging exhibits similar phenomena to the imaging on $\Gamma_1$.However, $N=5$ cannot sufficiently resolve the shape outline of $\Gamma_2$, even when $K$ is sufficiently large.

\begin{figure}[h]
\begin{center}
\includegraphics[width=0.325\textwidth]{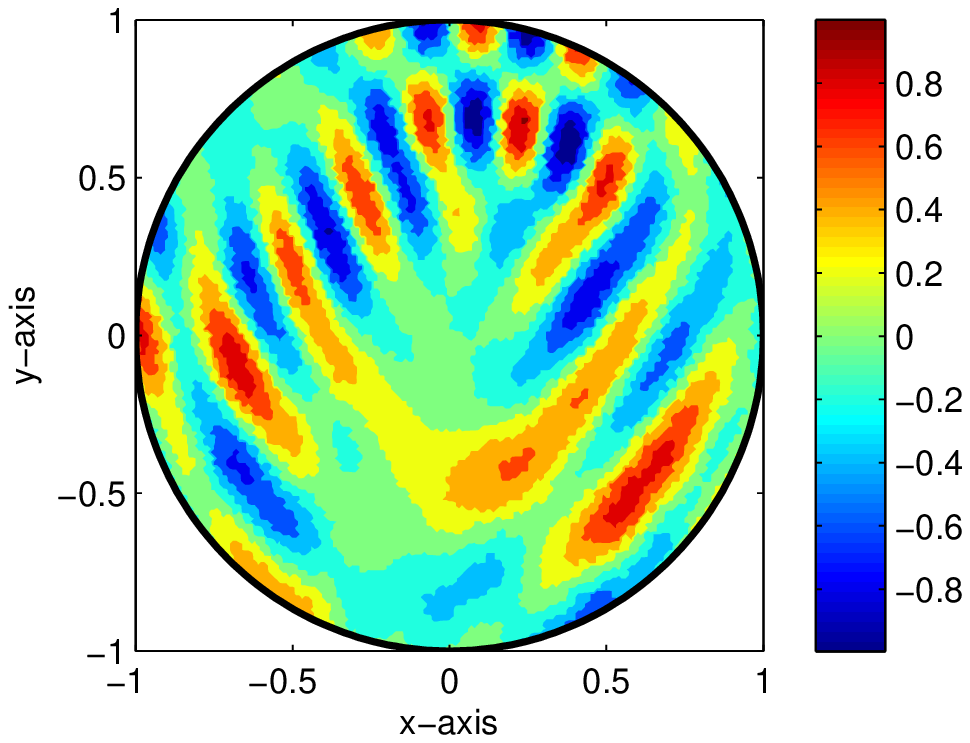}
\includegraphics[width=0.325\textwidth]{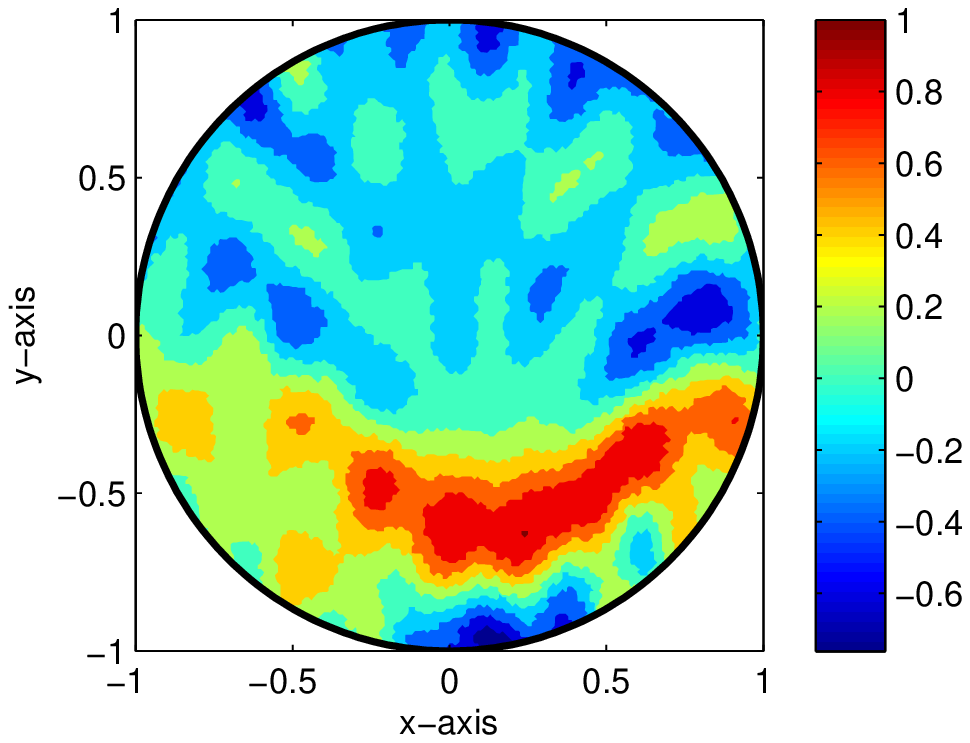}
\includegraphics[width=0.325\textwidth]{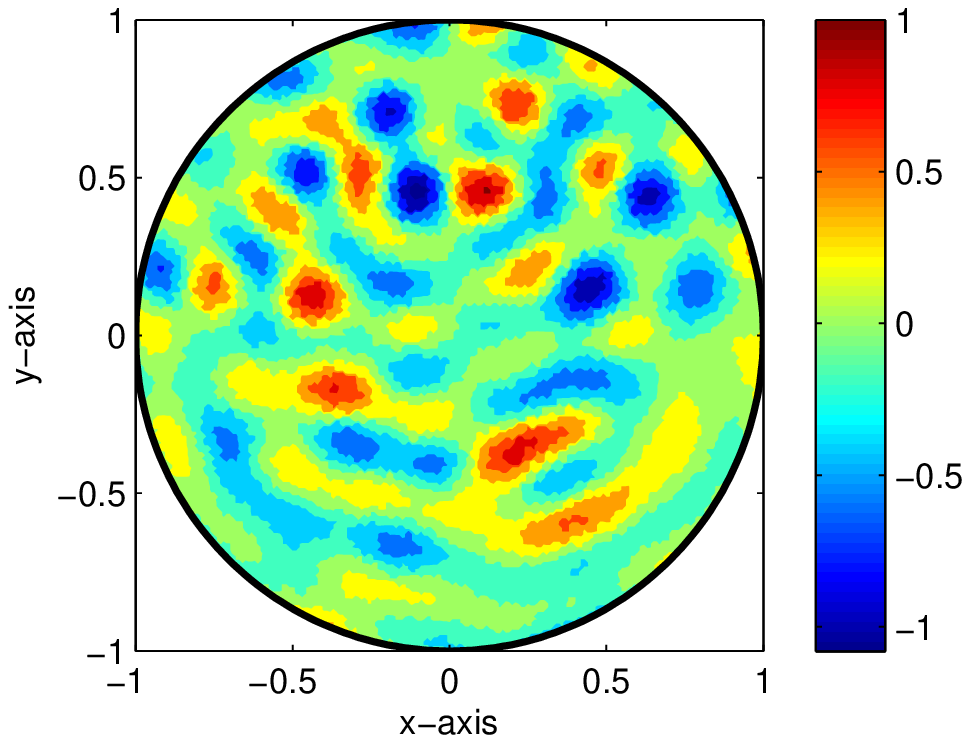}\\
\includegraphics[width=0.325\textwidth]{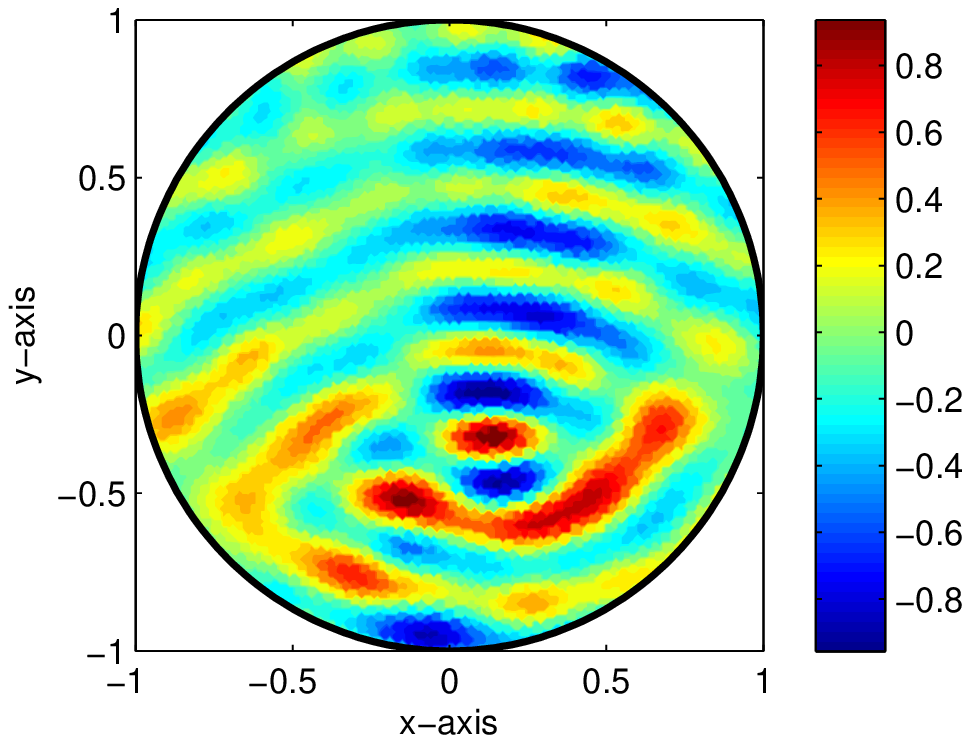}
\includegraphics[width=0.325\textwidth]{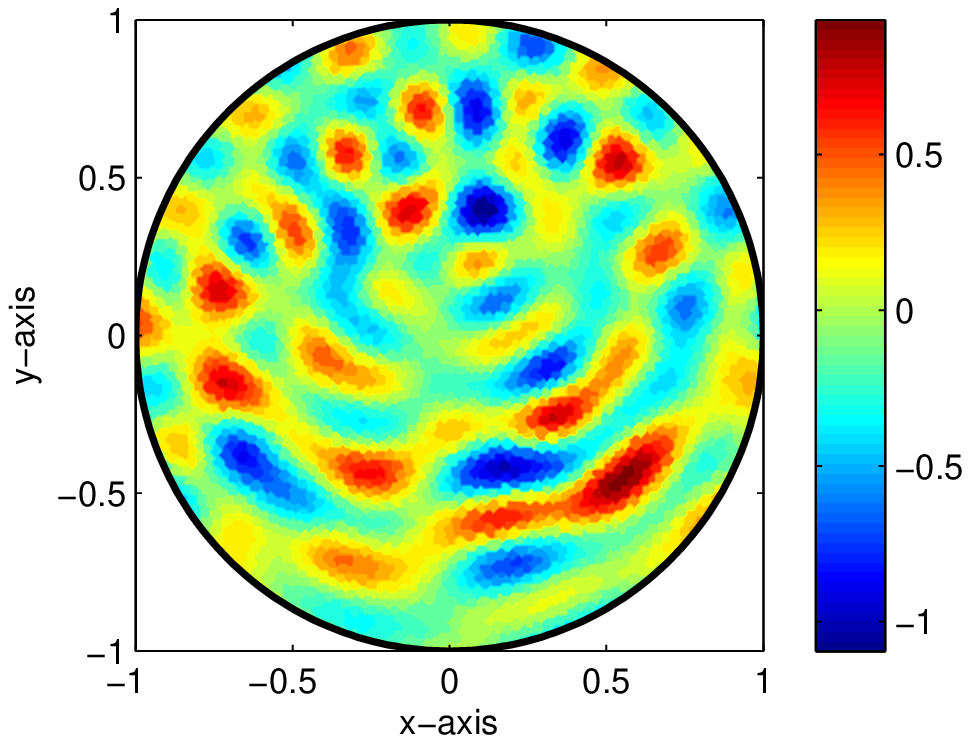}
\includegraphics[width=0.325\textwidth]{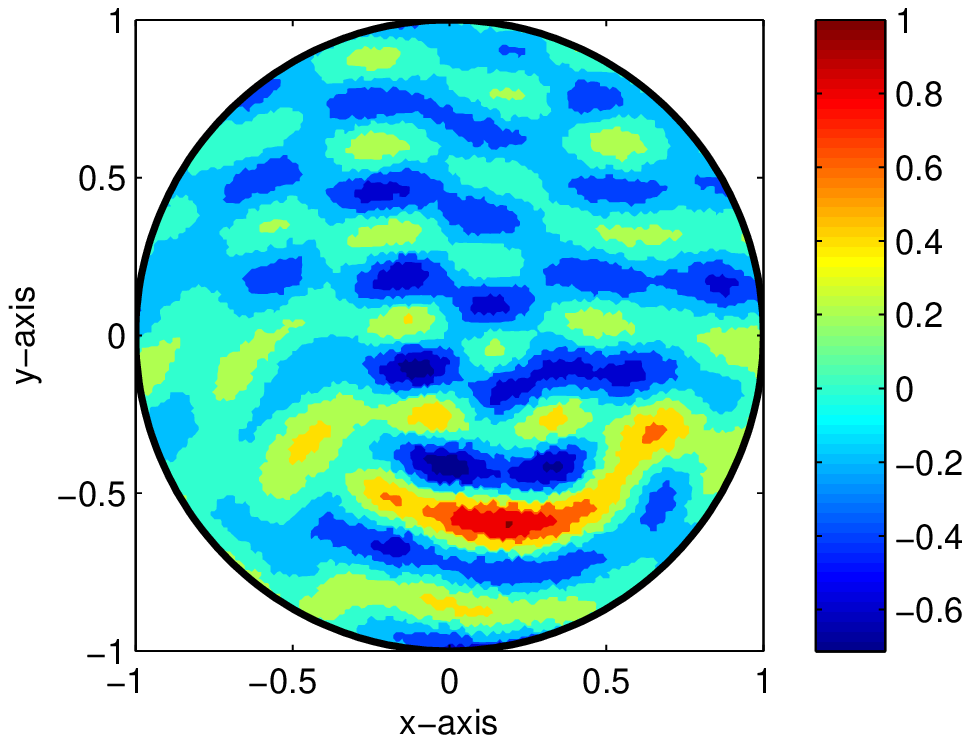}
\caption{Same as Figure \ref{Gamma1-Single}, but on a thin inhomogeneity of $\Gamma_2$.}\label{Gamma2-Single}
\end{center}
\end{figure}

\begin{figure}[h]
\begin{center}
\includegraphics[width=0.325\textwidth]{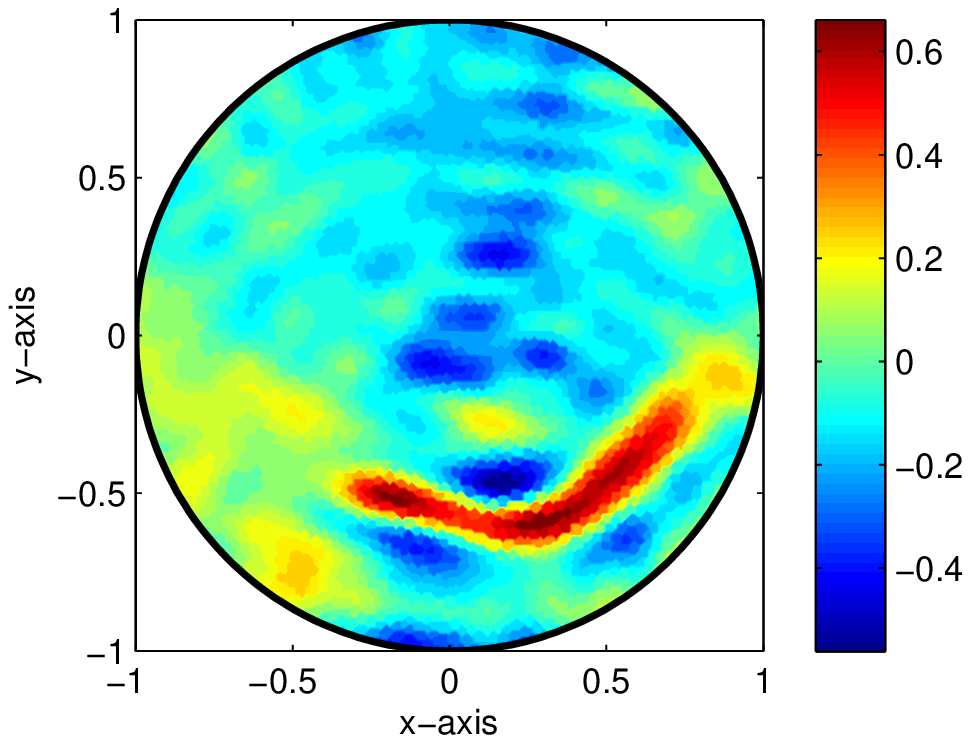}
\includegraphics[width=0.325\textwidth]{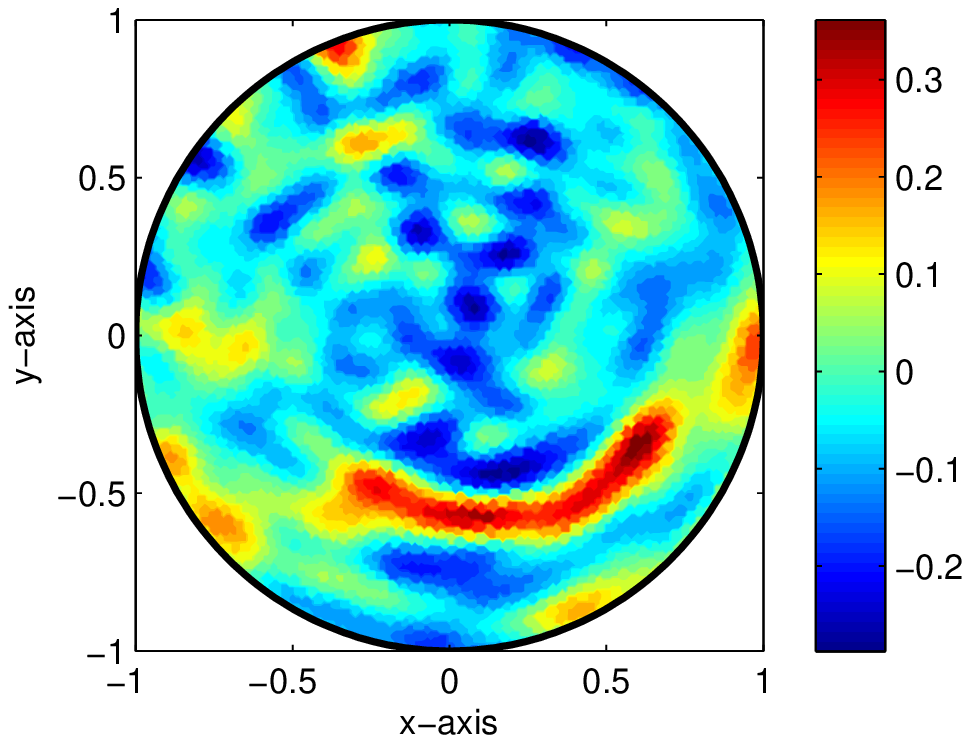}
\caption{Same as Figure \ref{Gamma1-Multi}, but on a thin inhomogeneity of $\Gamma_2$.}\label{Gamma2-Multi}
\end{center}
\end{figure}

The mathematical setting and numerical approach in \cite{P-TD1,P-TD3} are directly extendible to multiple inhomogeneities. Figure \ref{GammaM1-Single} displays maps of $\mathbb{E}_{\mathrm{SF}}(\mz;\omega)$ for $N=4$ and $N=5$ on multiple thin imaging inhomogeneities $\Gamma_1\cup\Gamma_2$ with $\eps_1=\eps_2=5$ and $\mu_1=\mu_2=5$. Unlike the single inhomogeneity cases, the true shapes of the inhomogeneities are difficult to discern when $N$ is small, because they are obscured by ghost replicas and artifacts. However, although $N=5$ is a poor choice in this case, the shapes of $\Gamma_1$ and $\Gamma_2$ are properly outlined in the map of $\mathbb{E}_{\mathrm{MF}}(\mz;K)$ when $N=4$ (see Figure \ref{GammaM1-Multi}).

\begin{figure}[h]
\begin{center}
\includegraphics[width=0.325\textwidth]{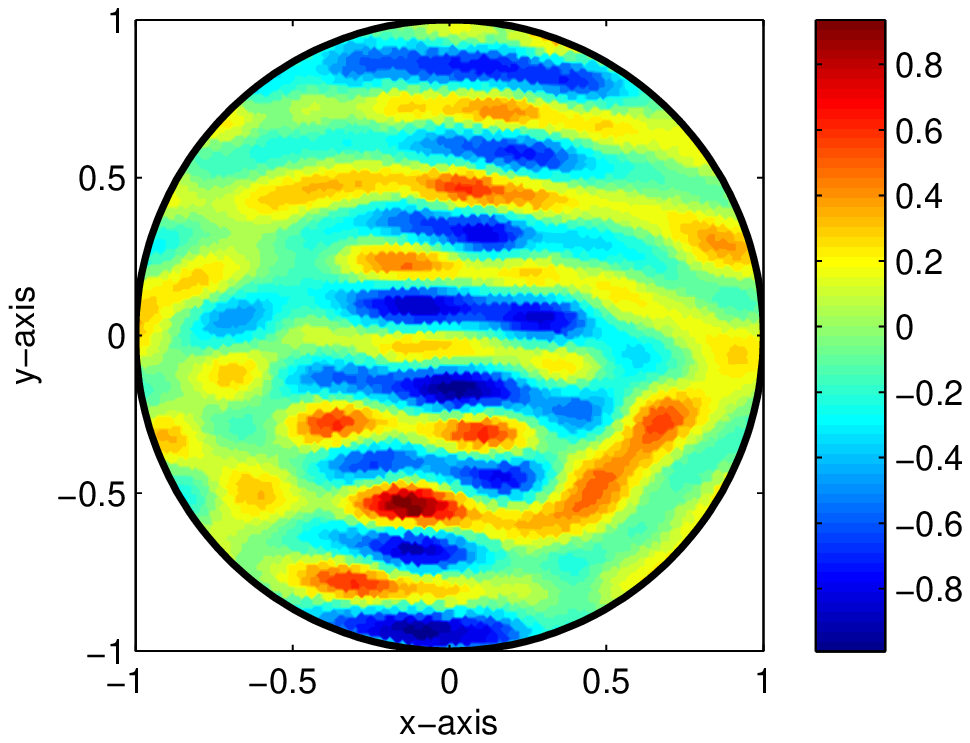}
\includegraphics[width=0.325\textwidth]{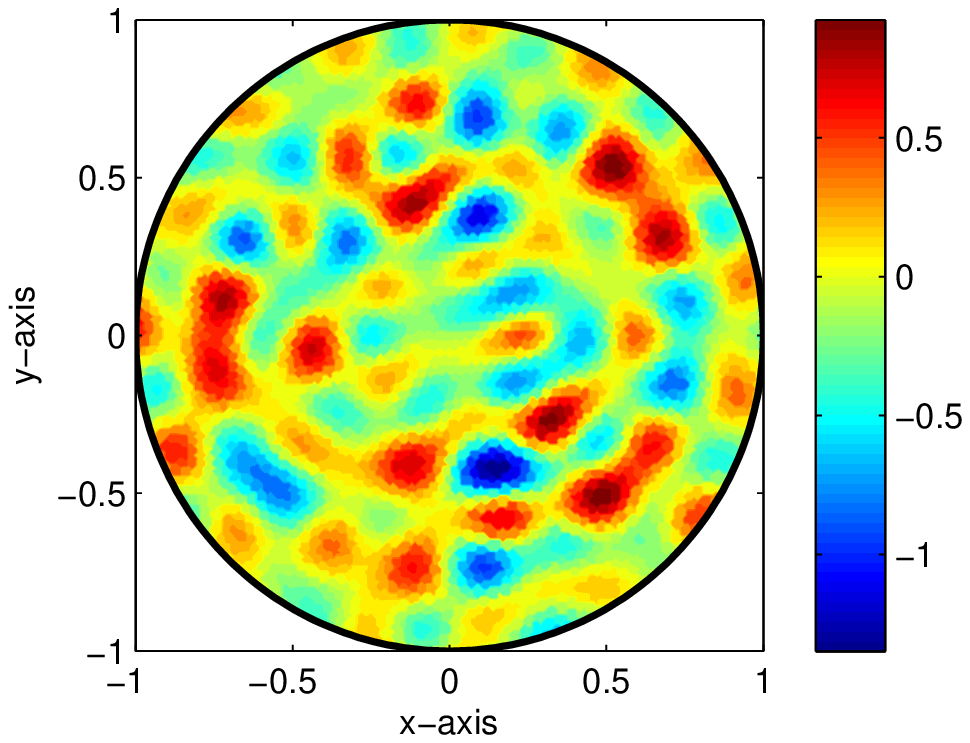}
\caption{Maps of $\mathbb{E}_{\mathrm{SF}}(\mz;2\pi/0.5)$ for $N=4$ (left) and $N=5$ (right) on thin inhomogeneities of $\Gamma_1$ and $\Gamma_2$ with the same permittivity and permeability.}\label{GammaM1-Single}
\end{center}
\end{figure}

\begin{figure}[h]
\begin{center}
\includegraphics[width=0.325\textwidth]{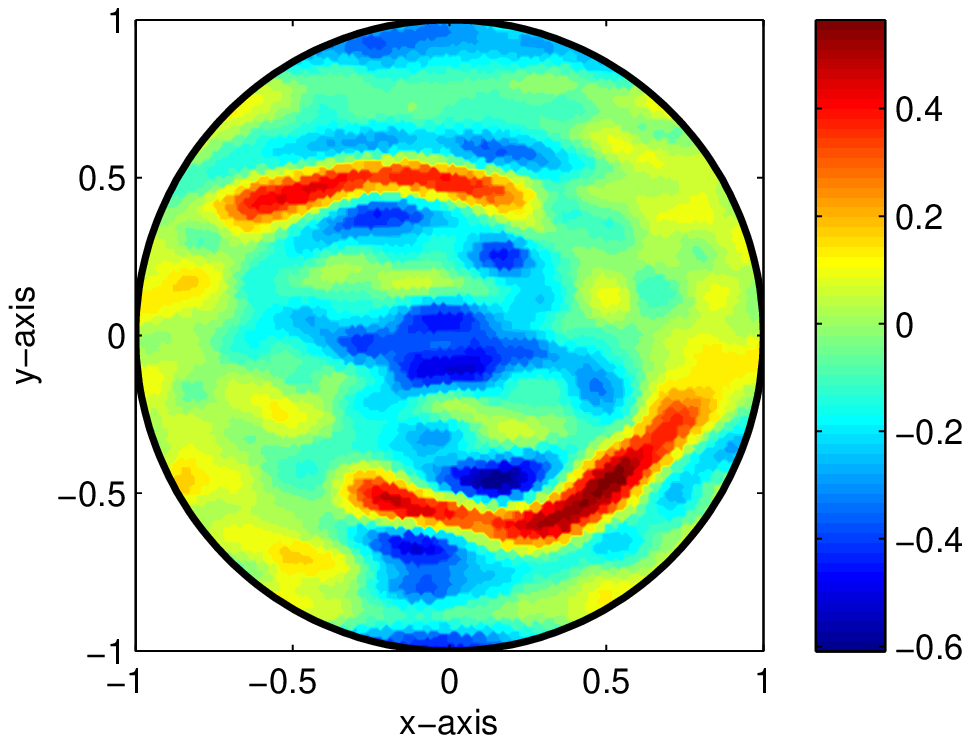}
\includegraphics[width=0.325\textwidth]{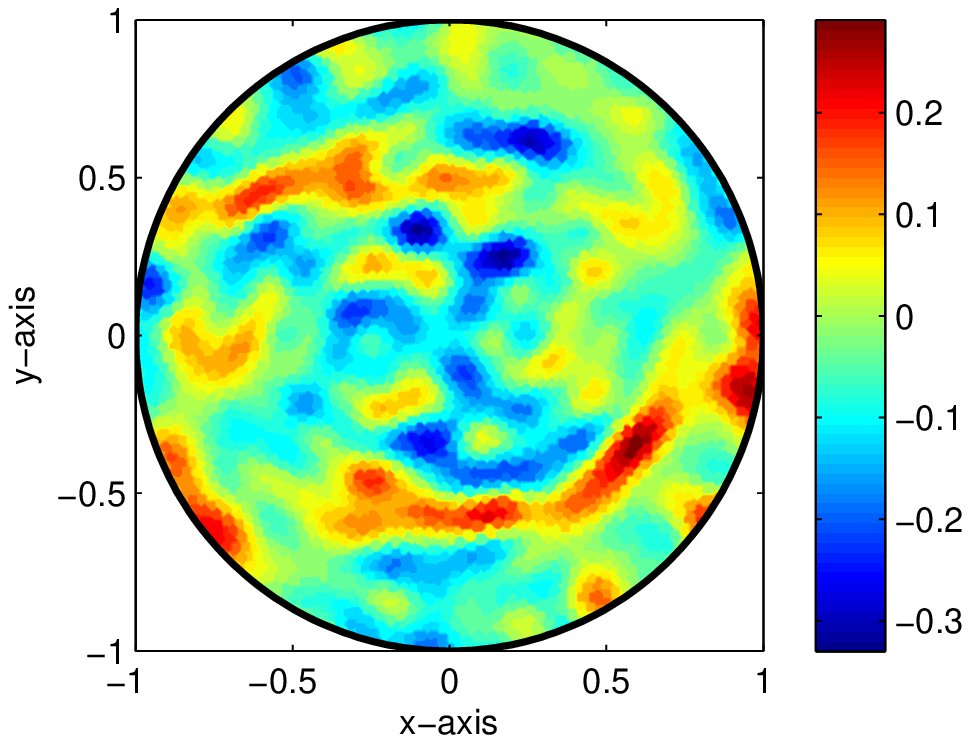}
\caption{Same as Figure \ref{GammaM1-Single} except maps of $\mathbb{E}_{\mathrm{MF}}(\mz;10)$.}\label{GammaM1-Multi}
\end{center}
\end{figure}

Figures \ref{GammaM2-Single} and \ref{GammaM2-Multi} present maps of $\mathbb{E}_{\mathrm{SF}}(\mz;\omega)$ and $\mathbb{E}_{\mathrm{MF}}(\mz;K)$, respectively, in the same configuration as the previous example but with different material properties ($\eps_1=\mu_1=5$ and $\eps_2=\mu_2=10$). Note that in the presence of two thin inhomogeneities, (\ref{Structure1}) and (\ref{Structure2}) can be re-written as
\begin{multline*}
  \mathbb{E}_{\mathrm{SF}}(\mz;\omega)\sim\sum_{n=1}^{N}\sum_{j=1}^{2}\int_{\sigma_j}\bigg[(\eps_j-\eps_0)+2\bigg(\frac{1}{\mu_j}-\frac{1}{\mu_0}\bigg)\vt_n\cdot\mt(\mx)
    +2\bigg(\frac{1}{\mu_0}-\frac{\mu_j}{\mu_0^2}\bigg)\vt_n\cdot\mn(\mx)\bigg]\\
    \times\left(J_0(\omega|\mx-\mz|)
 +2\sum_{m=1}^{\infty}(-1)^mJ_{2 m}(\omega|\mx-\mz|)\cos\{2 m(\theta_n-\phi_\mz)\}\right)d\sigma_j(\mx)
 \end{multline*}
  and
  \begin{multline*}
  \mathbb{E}_{\mathrm{MF}}(\mz;K)\sim\sum_{n=1}^{N}\sum_{j=1}^{2}\int_{\sigma_j}\bigg[(\eps_j-\eps_0)+2\bigg(\frac{1}{\mu}-\frac{1}{\mu_0}\bigg)\vt_n\cdot\mt(\mx)
    +2\bigg(\frac{1}{\mu_0}-\frac{\mu_j}{\mu_0^2}\bigg)\vt_n\cdot\mn(\mx)\bigg]\\
    \times\frac{1}{\omega_K-\omega_1}\left(\Lambda(|\mx-\mz|;K)
 +2\int_{\omega_1}^{\omega_K}\sum_{m=1}^{\infty}(-1)^mJ_{2 m}(\omega|\mx-\mz|)\cos\{2 m(\theta_n-\phi_\mz)\}d\omega\right)d\sigma_j(\mx),
 \end{multline*}
respectively. Therefore, if $\mz_1\in\Gamma_1$ and $\mz_2\in\Gamma_2$, then $\mathbb{E}_{\mathrm{SF}}(\mz_1;\omega)\leq\mathbb{E}_{\mathrm{SF}}(\mz_2;\omega)$ and $\mathbb{E}(\mz_1;K)\leq\mathbb{E}(\mz_2;K)$, i.e., the magnitude of $\Gamma_1$ will be much smaller than that of $\Gamma_2$ in the maps of $\mathbb{E}_{\mathrm{SF}}(\mz;\omega)$ and $\mathbb{E}_{\mathrm{MF}}(\mz;K)$. Hence, the shape of $\Gamma_1$ will be difficult to recognize when $N$ is small. Although the shape of $\Gamma_2$ is recognizable when $N=4$ and $N=5$, the shape is deteriorated by various large-magnitude artifacts when $N=5$, even when $K$ is sufficiently large (see Figure \ref{GammaM2-Multi}).

\begin{figure}[h]
\begin{center}
\includegraphics[width=0.325\textwidth]{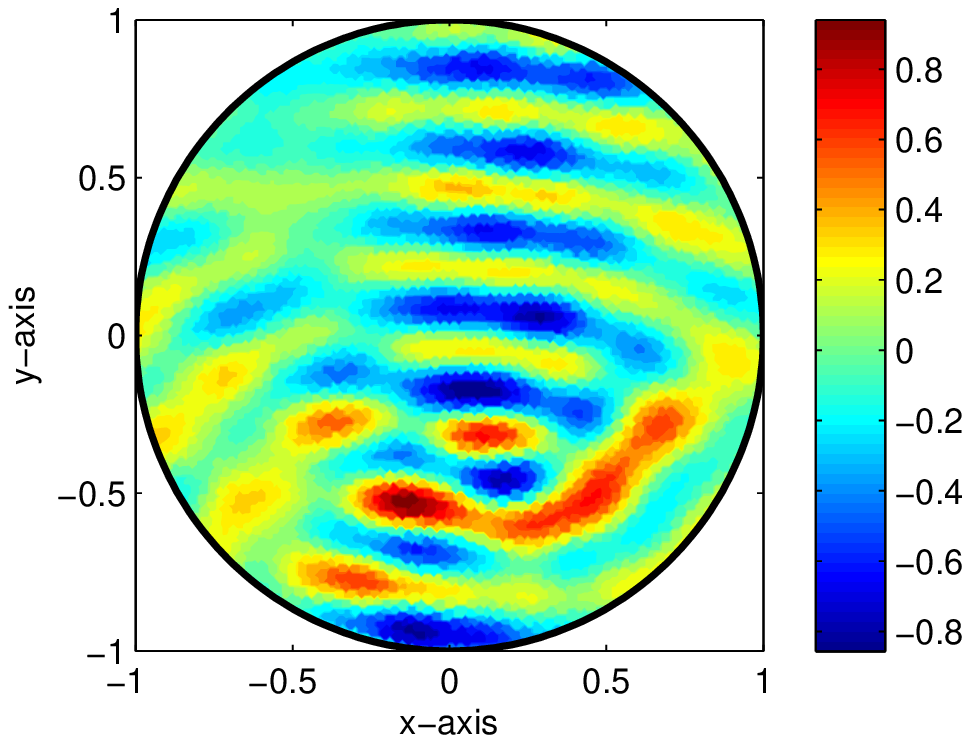}
\includegraphics[width=0.325\textwidth]{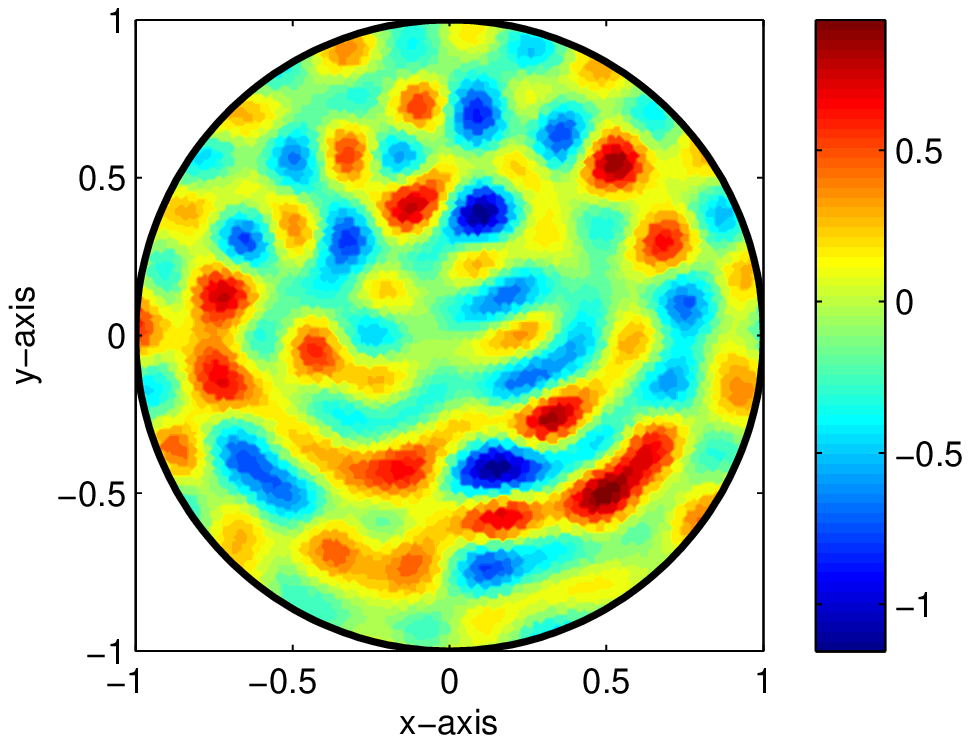}
\caption{Same as Figure \ref{GammaM1-Single}, but with different permittivities and permeabilities.}\label{GammaM2-Single}
\end{center}
\end{figure}

\begin{figure}[h]
\begin{center}
\includegraphics[width=0.325\textwidth]{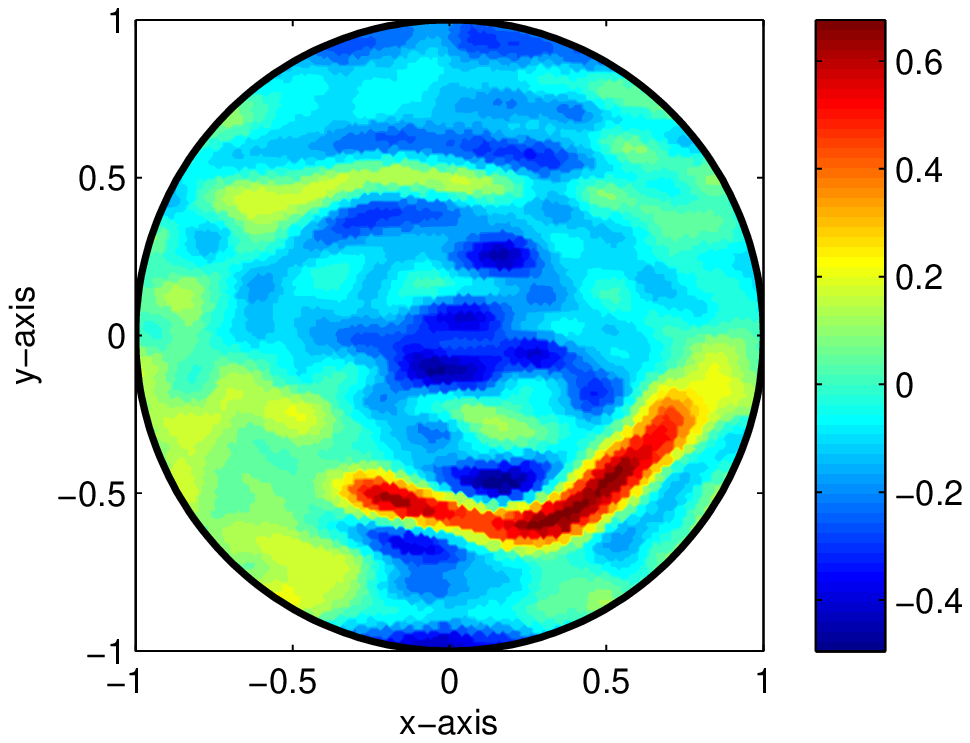}
\includegraphics[width=0.325\textwidth]{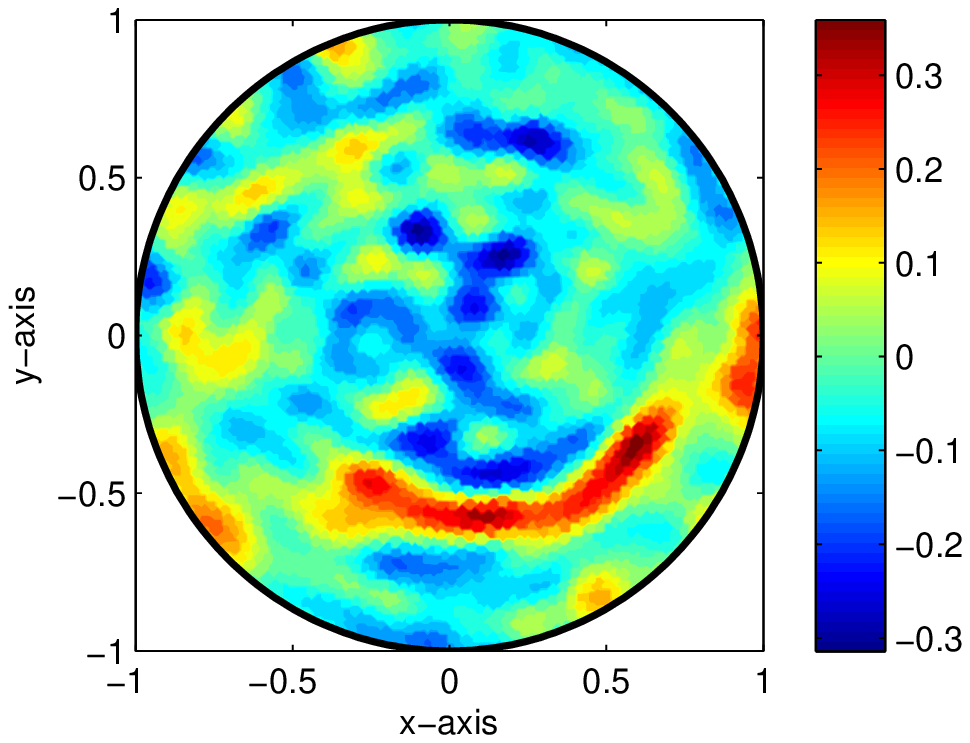}
\caption{Same as Figure \ref{GammaM2-Single} except maps of $\mathbb{E}_{\mathrm{MF}}(\mz;10)$.}\label{GammaM2-Multi}
\end{center}
\end{figure}

\section{Conclusion}\label{sec:5}
We applied the single- and multi-frequency topological derivative based on a non-iterative technique to the imaging of two-dimensional, thin penetrable inclusions embedded in a homogeneous domain. For this purpose, we varied the number of incident directions and related the topological derivative-based imaging function to an infinite series of Bessel functions of the first kind. From this relationship, we confirmed that successful imaging requires an even number (at least $4$) of incident directions. Moreover, the incident directions must be symmetrically distributed. The study also theoretically explains why small odd numbers of incident directions yield poor results.

Currently, our approach has limited ability in the imaging of multiple inhomogeneities with different permittivities and permeabilities. Improving this deficiency will be the focus of our future work.

\section*{Acknowledgments}
This research was supported by Basic Science Research Program through the National Research Foundation of Korea (NRF) funded by the Ministry of Education(No. NRF-2014R1A1A2055225) and the research program of Kookmin University in Korea.

\bibliographystyle{elsarticle-num-names}
\bibliography{../../../References}

\end{document}